\documentclass[pdflatex,sn-mathphys-num]{sn-jnl}

\usepackage{graphicx}%
\usepackage{multirow}%
\usepackage{amsmath,amssymb,amsfonts}%
\usepackage{amsthm}%
\usepackage{mathrsfs}%
\usepackage[title]{appendix}%
\usepackage{xcolor}%
\usepackage{textcomp}%
\usepackage{manyfoot}%
\usepackage{booktabs}%
\usepackage{algorithm}%
\usepackage{algorithmicx}%
\usepackage{algpseudocode}%
\usepackage{listings}%

\theoremstyle{thmstyleone}%
\newtheorem{theorem}{Theorem}[section]
\newtheorem{proposition}{Proposition}[section]
\newtheorem{lemma}{Lemma}[section]

\newtheorem{corollary}{Corollary}[section]
\newtheorem{definition}{Definition}[section]
\newtheorem{remark}{Remark}[section]
\newtheorem{problem}{Problem}[section]

\begin{document}

\title[An algebraic approach to Latin squares of prime power order by LPPs]{An algebraic approach to Latin squares of prime power order by local permutation polynomials}

\author*[1]{\fnm{Ra\'ul M.} \sur{Falc\'on}}\email{rafalgan@us.es}

\author[2]{\fnm{Jaime} \sur{Guti\'errez}}\email{jaime.gutierrez@unican.es}

\author[3]{\fnm{Jorge} \sur{Jim\'enez Urroz}}\email{jorge.urroz@upm.edu}

\affil*[1]{\orgdiv{Departamento de Matem\'atica Aplicada I}, \orgname{Universidad de Sevilla}, \orgaddress{\street{Avenida Reina Mercedes, 4 A}, \city{Seville}, \postcode{41012}, \country{Spain}}}

\affil[2]{\orgdiv{Departamento de Matem\'atica Aplicada y Ciencias de la Computaci\'on}, \orgname{Universidad de Cantabria}, \orgaddress{\street{Avenida Los Castros 44}, \city{Santander}, \postcode{39005}, \country{Spain}}}

\affil[3]{\orgdiv{Departamento de Matem\'atica e Inform\'atica Aplicadas a la Ingenier\'\i a Civil y Naval}, \orgname{Universidad Polit\'ecnica de Madrid}, \orgaddress{\street{Calle Profesor Aranguren S/N}, \city{Madrid}, \postcode{28040}, \country{Spain}}}

\abstract{Every Latin square of prime power order $q$ is uniquely described by a local permutation polynomial (LPP) in the polynomial ring $\mathbb{F}_q[x,y]$. Despite this equivalence, one may find in the literature only some preliminary results on the relationship among Latin squares and LPPs. This paper delves into this topic by showing how the coefficients of any LPP are identified with the zeros of an algebraic set over $\mathbb{F}_q$. This allows for an algebraic description of all Latin squares of order $q$ by means of a unique polynomial in $\mathbb{F}_q[x,y]$, whose coefficients satisfy the constraints defined by the algebraic set under consideration. In order to make much easier the construction of this polynomial, we also deal with the natural translation to LPPs of both notions of reduced and isotopic Latin squares. Our algebraic approach is readily adapted to identify both types of Latin squares. All of the above is constructively illustrated for $q\in\{4,5\}$. We finish our study with the natural translation to LPPs of both notions of complete mappings and orthomorphisms of quasigroups, showing their relationship with transversals and isotopisms of Latin squares.}

\keywords{Local permutation polynomial; permutation polynomial; Latin square; quasigroup; isotopism; complete mapping; orthomorphism.}

\pacs[MSC Classification]{11T06; 05B05; 20N05}

\maketitle

\section{Introduction}\label{sec:introduction}

Let $\mathbb{F}_q$ be the finite field of order a prime power $q$. It is well known \cite{Lidl1997} that any mapping from $\mathbb{F}_q$ to $\mathbb{F}_q$ is defined by a unique polynomial in the polynomial ring $\mathbb{F}_q[x]$ of degree lower than $q$. This is called a {\em permutation polynomial} (PP) if it acts as a permutation on $\mathbb{F}_q$. This kind of polynomials has relevant applications in cryptography, coding theory, and combinatorial design (see \cite{Hou2015} for a survey on PPs.) Furthermore, a polynomial $f$ in the polynomial ring $\mathbb{F}_q[x,y]$ is called a {\em local permutation polynomial} (LPP) if both polynomials $f(x,a)$ and $f(a,x)$ are PPs in $\mathbb{F}_q[x]$ for all $a\in\mathbb{F}_q$. From now on, we denote by $\mathcal{P}_q$ and $\mathcal{LP}_q$ the respective sets of PPs in $\mathbb{F}_q[x]$ and LPPs in $\mathbb{F}_q[x,y]$. Throughout this paper, we identify all functions $\mathbb{F}_q^2 \to \mathbb{F}_q$ with the interpolating polynomial, with degree less than $q$ in each variable.

Each polynomial in $\mathcal{LP}_q$ describes uniquely both a quasigroup and a Latin square of order $q$. Recall here that, if $\mathbb{F}_q$ is endowed with a binary operation $f$ such that both equations $f(s,a)=b$ and $f(a,s)=b$ have exactly one solution $s\in \mathbb{F}_q$ for any $a,b\in \mathbb{F}_q$, then the pair $(\mathbb{F}_q,f)$ is said to be a {\em quasigroup} of order $q$. Its multiplication table is a {\em Latin square} of the same order, which we denote by $L_f$. (That is, $L_f[a,b]:=f(a,b)$ for all $a,b\in\mathbb{F}_q$.) This is a $q\times q$ array such that each element in $\mathbb{F}_q$ appears exactly once per row and exactly once per column. In cryptography, quasigroups and Latin squares have been used to design error-correcting codes, cryptographic primitives, pseudorandom sequences with high period growth or image encryption algorithms, among other applications (see \cite{Shcherbacov2017}). From now on, we denote by $\mathcal{L}_q$ the set of Latin squares over $\mathbb{F}_q$. Mullen \cite{Mullen1981} (see also \cite{Mullen1980, Mullen1980a}) realized that every Latin square $L\in\mathcal{L}_q$ is uniquely identified with a polynomial $f_L\in\mathcal{LP}_q$, with $\mathrm{deg}(f_L)<2q-1$, so $f_L(a,b):=L[a,b]$ for all $a,b\in\mathbb{F}_q$. That is, $(L,f_L)$ is a quasigroup.

This equivalence among Latin squares, quasigroups and LPPs gives rise to a natural translation of related concepts and results. This paper delves into this topic by focusing on how the study of LPPs can provide some relevant information on the problems of counting, enumerating, and distributing Latin squares into isotopism classes. These problems have been completely solved only for Latin squares of order $n\leq 11$ (see \cite{Hulpke2011,Kolesova1990,McKay2007}). In the literature, one may find only some preliminary results on this topic \cite{Mullen1980, Mullen1980a, Lee1994, Diestelkamp2004, Gutierrez2023, Gutierrez2025,Gutierrez2025a,Hasan2026}. In particular, Mullen  \cite{Mullen1980} described  the set of coefficients of every polynomial in $\mathcal{LP}_p$, with $p$ an odd prime, as an algebraic set over $\mathbb{F}_p$. This algebraic set is in one-to-one correspondence with the set $\mathcal{L}_p$. In this regard, he proved that every polynomial in $\mathcal{LP}_2$ and $\mathcal{LP}_3$ is linear. More precisely,

\begin{equation}\label{eq_LPP2}
\mathcal{LP}_2=\bigcup_{s_{00}\in \mathbb{F}_2}\left\{x+y+s_{00}\right\}
\end{equation}
and
\begin{equation}\label{eq_LPP3}\mathcal{LP}_3=\bigcup_{a,b,c\in \mathbb{F}_3}\left\{s_{10}x+s_{01}y+s_{00}\colon\,  s_{10}^2s_{01}^2=1\right\}.
\end{equation}
As a consequence,
\[\mathcal{L}_2\cong \mathbb{F}_2 \hspace{1cm}\text{and}\hspace{1cm} \mathcal{L}_3\cong \left\{(a,b,c)\in\mathbb{F}_3^3\colon\,a^2b^2=1\right\}.\]
This paper deals with the following more general question. 

\begin{problem}\label{problem_algebraic set} Let $q$ be a prime power. Describe explicitly the sets of coefficients of any polynomial in $\mathcal{P}_q$ (respectively, $\mathcal{LP}_q$), as an algebraic set over $\mathbb{F}_q$. Use this algebraic set to count the number of polynomials in $\mathcal{LP}_q$.
\end{problem}

In this paper, we show how the existence of these algebraic sets allows the description of all Latin squares of order $q$ by a unique polynomial in $\mathbb{F}_q[x,y]$. This unified representation makes it much easier to deal with the problems of counting, enumerating, and classifying Latin squares. To the best of our knowledge, this paper is a pioneer in dealing with Problem \ref{problem_algebraic set} for LPPs. Concerning the set $\mathcal{P}_q$, in the literature one may find some partial results characterizing and counting some PPs with either few terms (see, for example, \cite{Hou2015a, Jarali2023, Tu2019, Ozbudak2023, Zhang2023}, and the references therein) or small degree. In this last regard, we have $ax+b\in\mathcal{P}_q$ for all $a,b\in\mathbb{F}_q$. In addition, Mollin and Small \cite{Mollin1987} proved the following two statements for PPs with degrees two and three.
\begin{enumerate}
    \item[(D2)] $ax^2+bx+c\in\mathcal{P}_q$, with $a\neq 0$, if and only if $b=0$ and $\mathrm{char}(\mathbb{F}_q)=2$.

    \item[(D3)] If $\mathrm{char}(\mathbb{F}_q)\neq 3$, then $ax^3+bx^2+cx+d\in\mathcal{P}_q$, with $a\neq 0$, if and only if $b^2=3ac$ and $q\equiv 2\, (\mathrm{mod}\, 3)$.
\end{enumerate}
Let $\mathcal{P}_{q,d}$ be the subset of polynomials of degree $d$ in $\mathcal{P}_q$. The problem of determining the number $N_q(d):=|\mathcal{P}_{q,d}|$ was introduced in \cite{Lidl1988}. In particular, $N_q(1)=q(q-1)$, and, if $q>2$, then $N_q(d)=0$ whenever $d$ divides $q-1$. Hence, $q-2$ (respectively, $2q-4$) is an upper bound for the degree of a polynomial in $\mathcal{P}_q$ (respectively, $\mathcal{LP}_q$), with $q>2$. The upper bound for LPPs decreases to $q-2$ if the polynomial describes an $m$-circulant Latin square (see \cite{Diestelkamp2004}). A formula for $N_p(p-2)$, with $p$ prime, was provided in \cite{Das2002} in terms of the permanent of a Vandermonde matrix whose entries are the $p^{\mathrm{th}}$ roots of unity. (A generalization for prime powers was provided in \cite{Kim2016}.) The number of polynomials in $\mathcal{P}_q$ of degree less than $q-2$ has been asymptotically bounded in \cite{Konyagin2002, Das2002, Kim2016}. This last number is explicitly only known for $q\leq 11$ (see \cite{Konyagin2002}). There is no study on the distribution of polynomials in $\mathcal{LP}_q$ with respect to their total degrees.

\vspace{0.2cm}

The paper is organized as follows. 
\begin{itemize}
    \item In Section \ref{sec:preliminaries}, we indicate some preliminary concepts and results on algebraic geometry and Latin squares that are used throughout the paper. 
    
    \item In Section \ref{sec:description}, we give the explicit description required in Problem \ref{problem_algebraic set} and we describe an algebraic geometry procedure to efficiently determine the constraints satisfied by the coefficients of any polynomial in $\mathcal{P}_q$ or $\mathcal{LP}_q$. We use this procedure to characterize all polynomials in $\mathcal{P}_4$, $\mathcal{P}_5$ and $\mathcal{LP}_4$. We conveniently modify our procedure to deal with polynomials of a given degree, symmetric LPPs and the so-called reduced LPPs, which we introduce as a natural translation of reduced Latin squares. In particular, we determine the values of $N_q(d)$ for all $q\leq 13$, and also the number of symmetric and reduced polynomials in $\mathcal{LP}_q$, with $q\in\{4,5\}$, of any given degree.
    
    \item In Section \ref{sec:symmetries}, we introduce the natural translation to LPPs of the notion of isotopism of Latin squares, which allows the distribution of LPPs into isotopism classes. In particular, the class containing the linear polynomial $x+y$ gives rise to what we call isolinear LPPs. We show that any polynomial that is not isolinear constitutes a solution of an open problem raised by Lee and Ko in \cite{Lee1994}. We also describe an isotopism that maps any polynomial in $\mathcal{L}_q$ to a reduced LPP. We illustrate for $q=4$ how this map facilitates the distribution of LPPs into isotopism classes.

    \item Finally, Section \ref{sec:Orthomorphism} deals with the natural translation to LPPs of the notions of complete mappings and orthomorphisms of Latin squares of quasigroups, showing their relationship with transversals and isotopisms of Latin squares. Winterhof  \cite{Winterhof2014} had already considered the generalization of these concepts using PPs. 
\end{itemize}

\section{Preliminaries}\label{sec:preliminaries}

In this section, we indicate some preliminary notions and results on algebraic geometry and transversals and isotopisms of Latin squares that we use throughout the paper. We refer the reader to \cite{Cox2015,Keedwell2015} for two more details on these topics. In addition, we refer the reader to \cite{Montgomery2024,Wanless2011} for a pair of comprehensive surveys on transversals of Latin squares.

\subsection{Algebraic geometry}

Let $\mathbb{F}_q[X]$ be a multivariate polynomial ring over the finite field $\mathbb{F}_q$, with $q$ a prime power, on a set of variables $X=\{x_1,\ldots,x_n\}$. An {\em ideal} of $\mathbb{F}_q[X]$ is any subset $I\subseteq \mathbb{F}_q[X]$ such that $0\in I$; $f+g\in I$ for all $f,g\in I$; and $fg\in I$ for all $(f,g)\in I\times \mathbb{F}_q[X]$. Its {\em algebraic set} is the set of zeros
\[\mathcal{Z}(I):=\left\{(s_1,\ldots,s_n)\in\mathbb{F}_q^n\colon f(s_1,\ldots,s_n)=0, \text{ for all } f\in I\right\}.\]
The ideal $I$ is {\em zero-dimensional} if and only if $\dim_{\mathbb{F}_q}\left(\mathbb{F}_q[X] / I\right)<\infty$. This is {\em radical} if every polynomial $f\in \mathbb{F}_q[X]$ such that $f^m\in I$ for some positive integer $m$ belongs to $I$. If the ideal $I$ is zero-dimensional and radical, then $\mathcal{Z}(I)$ contains precisely $\dim_{\mathbb{F}_q}\left(\mathbb{F}_q[X] / I\right)$ distinct elements (see \cite[Theorem 3.7.19]{Kreuzer2000}). Its computation follows from  the reduced Gr\"obner basis of the ideal, which also makes easier the computation of $\mathcal{Z}(I)$. We remind the concept of Gr\"obner basis in the following paragraph.

The ideal of $\mathbb{F}_q[X]$ that is {\em generated by} the polynomials $f_1,\ldots,f_n\in \mathbb{F}_q[X]$ is
\[\langle\,f_1,\ldots,f_n\,\rangle:=\left\{\sum_{i=1}^n f_i\cdot g_i\colon\, g_i\in\mathbb{F}_q[X], \text{ for all } i\right\}.\]
The largest monomial of a polynomial $f\in I$, with respect to a given monomial term ordering $\preceq$ is its {\em leading monomial}. The ideal generated by all the leading monomials in $I$ is the {\em initial ideal} $I_{\preceq}$. A generating set of $I$ whose leading monomials generate $I_{\preceq}$ is a {\em Gr\"obner basis} of $I$. It is {\em reduced} if all its polynomials are monic, and no monomial in a generator is generated by the leading monomials of the other generators. This reduced Gr\"obner basis is unique. Its computation is extremely sensitive to the number of variables, and the length and degree of the generators \cite{Hashemi2011}.

The computation of reduced Gr\"obner bases and dimensions of quotient rings in this paper has always been done by using our own library {\em lpp.lib} in the open computer algebra system for polynomial computations {\sc Singular} \cite{Decker2025}. This library is available as supplementary material to the paper.

\subsection{Transversals and isotopisms of Latin squares}

Every Latin square $L\in\mathcal{L}_q$, with entries in $\mathbb{F}_q$, is uniquely determined by its set of entries
\[\mathrm{Ent}(L):=\left\{\left(a,b,L[a,b]\right)\colon\,a,b\in\mathbb{F}_q\right\}.\]
A {\em transversal} of $L$ is any subset $S\subset\mathrm{Ent}(L)$ formed by $q$ entries such that no two entries belong to the same row or the same column and no two entries contain the same symbol. Transversals of Latin squares are uniquely related to complete mappings and orthomorphisms of quasigroups, whose notions we recall here by means of the so-called Hadamard quasigroup product (see \cite{Falcon2023, Falcon2025}). The {\em Hadamard quasigroup product} of two permutations $g_1$ and $g_2$ in the symmetric group $\mathcal{S}_q$ on $\mathbb{F}_q$ with respect to a quasigroup $(\mathbb{F}_q,f)$ is a map $\odot_f(g_1,g_2):\mathbb{F}_q\to \mathbb{F}_q$ that is defined so that

\begin{equation}\label{eq_hqp_multiary}
\odot_f(g_1,g_2)(a):=f(g_1(a),\,g_2(a))
\end{equation}
for all $a\in\mathbb{F}_q$. If $\mathrm{Id}$ is the trivial permutation on $\mathbb{F}_q$, then a {\em complete mapping} of the quasigroup $(\mathbb{F}_q,f)$ is any permutation $g\in\mathcal{S}_q$ such that $\odot_f(\mathrm{Id},g)\in\mathcal{S}_q$. The latter is called an {\em orthomorphism} of the quasigroup. Lemma 4.2 in \cite{Falcon2025}  (see also Theorem 6.2 in \cite{Wanless2011}) ensures the following result.

\begin{lemma}\label{lemma_transversals} Let $(\mathbb{F}_q,f)$ be a quasigroup. A subset $S\subset \mathrm{Ent}(L_f)$ is a transversal of the Latin square $L_f\in\mathcal{L}_q$ if and only if there is a complete mapping $g$ of the quasigroup such that 
\[S=\left\{(a,g(a),\odot_f(\mathrm{Id}
,g)(a))\colon\, a\in \mathbb{F}_q \right\}.\]
\end{lemma}

\vspace{0.2cm}

Niederreiter  and Robinson \cite{Niederreiter1982} proved that every polynomial in $\mathcal{P}_q$ describing a complete mapping of the additive group $(\mathbb{F}_q,+)$, with $q>3$ an odd prime, has degree $\leq q-3$. Then, they determined all these polynomials having degree $<6$, as well as those of degree six for finite fields of order prime to six. 

\vspace{0.2cm}

Now, we recall how the symmetric group $\mathcal{S}_q$ acts on the set $\mathcal{L}_q$. An {\em isotopism} in $\mathcal{L}_q$ is every triple $\Theta:=(\theta_1,\theta_2,\theta_3)\in\mathcal{S}_q\times \mathcal{S}_q\times\mathcal{S}_q$. The {\em isotope} $L^{\Theta}\in \mathcal{L}_q$ of a Latin square $L\in \mathcal{L}_q$ is defined so that
\begin{equation}\label{eq_LS_isotopic}
\mathrm{Ent}\left(L^{\Theta}\right):=\left\{(\theta_1(a),\theta_2(b),\theta_3(c))\colon\, (a,b,c)\in\mathrm{Ent}(L)\right\}.
\end{equation}
That is, $\theta_1$ permutes the rows of $L$, $\theta_2$ permutes its columns, and $\theta_3$ permutes the symbols appearing in its cells. It is also said that $L$ and $L^\Theta$ are {\em isotopic}.  The isotopism $\Theta$ is {\em principal} if $\theta_3$ is the trivial permutation on $\mathbb{F}_q$. If this is the case, then $L$ and $L^\Theta$ are said to be {\em principal isotopic}. Furthermore, if $\theta_1=\theta_2=\theta_3$, then we get an {\em isomorphism} from $L$ to $L^\Theta$, which are then said to be {\em isomorphic}.  Being (principal) isotopic or isomorphic are equivalence relations among Latin squares.  Finally, if $L^{\Theta}=L$, then $\Theta$ is called an {\em autotopism} of $L$ (an {\em automorphism} if $\theta_1=\theta_2=\theta_3$). The set of autotopisms of $L$ has a group structure with the composition of permutations. All these notions are readily translated to quasigroups (see \cite{Albert1943}) and hence to LPPs. Section \ref{sec:symmetries} focuses on this last aspect.
 
\section{Algebraic sets describing PPs and LPPs}\label{sec:description}

This section deals with Problem \ref{problem_algebraic set}. More precisely, we show how both sets $\mathcal{P}_q$ and $\mathcal{LP}_q$, with $q>3$ a prime power, are identified with the algebraic sets of a pair of ideals over $\mathbb{F}_q$. (The case $q\in \{2,3\}$ was already considered by Mullen \cite{Mullen1980}.) The keystone here is the ideal associated with $\mathcal{P}_q$. It is due to the fact that every polynomial $f\in\mathcal{LP}_q$ is uniquely determined by the set of $q^2$ polynomials $\{f(x,c),\,f(c,x)\colon\,c\in\mathbb{F}_q\}\subseteq \mathcal{P}_q$. Thus, the ideal associated with $\mathcal{LP}_q$ has to be related to that associated with $\mathcal{P}_q$.

Let us consider the sets of variables
\[X_q:=\left\{x_i\colon\, i\in \{0,\ldots q-2\}\right\} \hspace{1cm}\text{and}\hspace{1cm} \overline{X}_q:=\left\{x_{ij}\colon\, i,j\in \{0,\ldots q-2\}\right\}.\]
In what follows, we show how the set $\mathcal{P}_q$ (respectively, $\mathcal{LP}_q$) is identified with the algebraic set of an ideal in $\mathbb{F}_q[X_q]$ (respectively, $\mathbb{F}_q[\overline{X}_q]$). To this end, we conveniently use the algebraic approach described by Bayer \cite{Bayer1982} (see also \cite{Adams1994}) for solving the graph coloring problem, which has later been used in the literature \cite{Falcon2007, Falcon2013, Falcon2015, Falcon2018, Falcon2020} for enumerating and classifying Latin squares and some related structures. More precisely, we define the function
\[F(x):=x^{q-1}-1.\]
Note here that every zero of the equation $F(x)=0$ belongs to $\mathbb{F}_q\setminus\{0\}$, and hence the equation $F(x-y)=0$ implies $x\neq y$. Then, for each $a,b,c\in \mathbb{F}_q$, with $a\neq b$, we define the following polynomials.
\[F_{a,b}:=F\left(\sum_{i=0}^{q-2}(a^i-b^i)x_i\right)\in \mathbb{F}_q[X_q]\]
\[G^{\mathrm{row}}_{a,b,c}:=F\left(\sum_{i,j=0}^{q-2}(a^i-b^i)c^jx_{ji}\right)\in \mathbb{F}_q[\overline{X}_q]\]
\[G^{\mathrm{col}}_{a,b,c}:=F\left(\sum_{i,j=0}^{q-2}(a^i-b^i)c^jx_{ij}\right)\in \mathbb{F}_q[\overline{X}_q]\]
Note that the polynomial $G^{\mathrm{row}}_{a,b,c}$ (respectively, $G^{\mathrm{col}}_{a,b,c}$) coincides with $F_{a,b}$ once each variable $x_i$ is replaced by the polynomial $\sum_{j=0}^{q-2}c^jx_{ji}$ (respectively,  $\sum_{j=0}^{q-2}c^jx_{ij}$). This fact will play a relevant role in Remark \ref{remak_affine} to make our algebraic approach more efficient.

\begin{lemma}\label{lemma_affine} The next statements hold.
\begin{enumerate}
    \item[a)] The set $\mathcal{P}_q$ is identified with the algebraic set of the following zero-dimensional radical ideal in $\mathbb{F}_q[X_q]$.
\[\mathcal{I}_q:=\left\langle\,F_{a,b}\colon\, a,b\in\mathbb{F}_q,\, a\neq b\,\right\rangle + \langle\,\,x_i ^q-x_i\colon\, i\in\{0,\ldots,q-2\}\rangle\] 
As a consequence, $|\mathcal{P}_q|=|\mathcal{Z}(\mathcal{I}_q)|=\dim_{\mathbb{F}_q}(\mathbb{F}_q[X_q]/\mathcal{I}_q)$.

\vspace{0.1cm}

    \item[b)] The set $\mathcal{LP}_q$ is identified with the algebraic set of the following zero-dimensional radical ideal in $\mathbb{F}_q[\overline{X}_q]$.
\[\mathcal{J}_q:= \left\langle\,G^{\mathrm{row}}_{a,b,c},\, G^{\mathrm{col}}_{a,b,c} \colon\, a,b,c\in\mathbb{F}_q,\, a\neq b\,\right\rangle+ \langle\,\,x_{ij} ^q-x_{ij}\colon\, i,j\in\{0,\ldots,q-2\}\rangle\]
As a consequence, $|\mathcal{LP}_q|=|\mathcal{Z}(\mathcal{J}_q)|=\dim_{\mathbb{F}_q}(\mathbb{F}_q[\overline{X}_q]/\mathcal{J}_q)$.
\end{enumerate}
\end{lemma}

\begin{proof} Each zero $S:=(s_0,\ldots,s_{q-2})\in \mathcal{Z}(\mathcal{I}_q)$ is uniquely identified with the polynomial 
\[f_S:=\sum_{i=0}^{q-2}s_ix^i \in \mathbb{F}_q[x].\]
We claim $f_S\in\mathcal{P}_q$ because, for each pair of elements $a,b\in \mathbb{F}_q$, with $a\neq b$, the polynomial $F_{a,b}$ implies that $f_S(a)\neq f_S(b)$. Since no polynomial in $\mathcal{P}_q$ has degree a divisor of $q-1$ (see \cite{Lidl1997}), the first statement holds.

Similarly, each zero $\overline{S}\in \mathcal{Z}\left(\mathcal{J}_q\right)$ is uniquely identified with the polynomial 
\[f_{\overline{S}}:=\sum_{i,j=0}^{q-2}s_{ij}x^iy^j\in \mathbb{F}_q[x,y],\]
where each $s_{ij}$ is the component of $\overline{S}$ related to the variable $x_{ij}$. The second statement holds readily from (a) once we observe that $G^{\mathrm{row}}_{a,b,c}=0$ (respectively, $G^{\mathrm{col}}_{a,b,c}=0$) implies that $f_{\overline{S}}(c,x)\in \mathcal{P}_q$ (respectively, $f_{\overline{S}}(x,c)\in \mathcal{P}_q$) for all $c\in \mathbb{F}_q$.

The respective consequences hold from the fact that both ideals are zero-dimensional and radical. They are zero-dimensional because of the presence of all the generators $x_i^q-x$ in $\mathcal{I}_q$, and $x_{ij}^q-x_{ij}$ in $\mathcal{J}_q$, which make their respective zeros to take values in $\mathbb{F}_q$. They are radical from Seidenberg's Lemma (see \cite[Proposition 3.7.15]{Kreuzer2000}). 
\end{proof}

\vspace{0.2cm}

From now on, we denote by $\mathcal{LP}_{q,d}$ the subset of polynomials of total degree $d$ in $\mathcal{LP}_q$. The following result shows how the ideals described in Lemma \ref{lemma_affine} can be conveniently modified to identify both subsets $\mathcal{P}_{q,d}$ and $\mathcal{LP}_{q,d}$ with the algebraic sets of a pair of ideals over $\mathbb{F}_q$. Its proof is similar to that of the aforementioned lemma.

\begin{lemma}\label{lemma_affine_degrees} The next statements hold.
\begin{enumerate}
    \item[a)] The set $\mathcal{P}_{q,d}$, with $d\leq q-2$, is identified with the algebraic set of the following zero-dimensional radical ideal in $\mathbb{F}_q[X_q]$.
\[\mathcal{I}_{q,d}:=\mathcal{I}_q + \left\langle\,x_i\colon\, i>d\,\right\rangle + \left\langle\,x_d^{q-1}-1\,\right\rangle\] 
In particular, $N_q(d)=|\mathcal{Z}(I_{q,d})|=\dim_{\mathbb{F}_q}(\mathbb{F}_q[X_q]/I_{q,d})$.

\vspace{0.2cm}

    \item[b)] The set $\mathcal{LP}_{q,d}$ is identified with the algebraic set of the following zero-dimensional radical ideal in $\mathbb{F}_q[\overline{X}_q]$.
\[\mathcal{J}_{q,d}:= \mathcal{J}_q + \left\langle\,x_{ij}\colon\, i+j>d\,\right\rangle + \left\langle\,\prod_{i+j=d}(x_{ij}^{q-1}-1)\,\right\rangle\]
As a consequence, $|\mathcal{LP}_{q,d}|=|\mathcal{Z}(\mathcal{J}_{q,d})|=\dim_{\mathbb{F}_q}(\mathbb{F}_q[\overline{X}_q]/\mathcal{J}_{q,d})$.
\end{enumerate}
\end{lemma}

\vspace{0.2cm}

\begin{remark}\label{remak_affine} The reduced Gr\"obner basis of $\mathcal{I}_q$ in Lemma \ref{lemma_affine} establishes some constraints that must be satisfied by the coefficients of every polynomial in $\mathcal{P}_q$. We can use these constraints to compute much more efficiently the reduced Gr\"obner basis of the ideal $\mathcal{J}_q$ of the aforementioned lemma. To this end, note that, for each $c\in\mathbb{F}_q$, the polynomials 
\[f_{\overline{S}}(x,c)=\sum_{i=0}^{q-2}\left(\sum_{j=0}^{q-2}s_{ij}c^j\right)x^i \hspace{1cm} \text{ and } \hspace{1cm} f_{\overline{S}}(c,y)=\sum_{i=0}^{q-2}\left(\sum_{j=0}^{q-2}s_{ji}c^j\right)y^i\]
are PPs. In practice, for each $c\in \mathbb{F}_q$, we remove from $\mathcal{J}_q$ both sets of generators
\[\left\{G^{\mathrm{row}}_{a,b,c}\colon\, a,b\in \mathbb{F}_q,\,a\neq b\right\}\hspace{1cm} \text{ and } \hspace{1cm}\left\{G^{\mathrm{col}}_{a,b,c}\colon\, a,b\in \mathbb{F}_q,\,a\neq b\right\}.\] 
Instead of the first set (respectively, the second one), we include the reduced Gr\"obner basis of the ideal $\mathcal{I}_q$ by replacing each variable $x_i$ in that basis by the expression 
\[\sum_{j=0}^{q-2} x_{ij}c^j \hspace{1cm} \text{(respectively, } \hspace{1cm} \sum_{j=0}^{q-2} x_{ji}c^j).\] (Recall here that each element $s_{ij}$ in $f_{\overline{S}}$ is associated to the variable $x_{ij}$.) In particular, for $c=0$, we replace each variable $x_i$ by the variable $x_{i0}$ (respectively, $x_{0i}$). This gives rise to a new zero-dimensional radical ideal in $\mathbb{F}_q[\overline{X}_q]$ whose algebraic set coincides with that of $\mathcal{J}_q$. Calculating the reduced Gr\"obner basis of this new ideal is much more efficient than calculating that of $\mathcal{J}_q$, because it avoids repeating computations.
\end{remark}

\vspace{0.2cm}

We illustrate this approach by characterizing all polynomials in $\mathcal{P}_4$ and $\mathcal{P}_5$. This characterization is consistent with statements (D2) and (D3) in the introductory section. Then we decompose $\mathcal{LP}_4$ into two disjoint subsets that, as we prove in Section \ref{sec:symmetries}, separately describe the two isotopism classes of Latin squares of order four.

\begin{proposition}\label{proposition_P4} It is verified that
\[\mathcal{P}_4=\mathcal{P}_{4,1}\sqcup \mathcal{P}_{4,2}=\bigcup_{s_0,s_1,s_2\in\mathbb{F}_4}\{s_2x^2+s_1x+s_0\colon\,s_1^3+s_2^3=1\},\]
where
\[\mathcal{P}_{4,1}=\bigcup_{s_0,s_1\in\mathbb{F}_4}\{s_1x+s_0\colon\, s_1^3=1\} \hspace{1cm}\text{and}\hspace{1cm}
\mathcal{P}_{4,2}=\bigcup_{s_0,s_2\in\mathbb{F}_4}\{s_2x^2+s_0\colon\, s_2^3=1\}.\]
\end{proposition}

\begin{proof} The reduced Gr\"obner basis of the ideal $I_4\subset \mathbb{F}_4[X_4]$ described in Lemma \ref{lemma_affine}, with respect to the degree inverse lexicographical order, is the set
\begin{equation}\label{equation_G4}
\left\{x_1^3 + x_1^2 x_2 + x_1x_2^2 + x_2^3 + 1,\hspace{0.3cm}x_1x_2,\hspace{0.3cm} x_2^4+x_2\right\}.
\end{equation}
This is also the reduced Gr\"obner basis of the ideal $\langle\, x_1^3+x_2^3-1,\hspace{0.3cm}x_1x_2\,\rangle$, which therefore coincides with $I_4$. Since the algebraic set of this last ideal coincides with that of the ideal generated only by the polynomial $x_1^3+x_2^3-1$, Lemma \ref{lemma_affine} implies that
\[\mathcal{P}_4\cong\mathcal{Z}(I_4)=\left\{(s_0,s_1,s_2)\in\mathbb{F}_4^3\colon\,s_1^3+s_2^3=1\right\}.\]
The characterization of the polynomials in both subsets $\mathcal{P}_{4,1}$ and $\mathcal{P}_{4,2}$ follows readily.
\end{proof}

\vspace{0.2cm}

\begin{proposition}\label{proposition_P5} It is verified that
\[\mathcal{P}_5=\mathcal{P}_{5,1}\sqcup \mathcal{P}_{5,3}=\bigcup_{s_0,s_1,s_2,s_3\in\mathbb{F}_5}\{s_3x^3+s_2x^2+s_1x+s_0\colon\, (s_1+s_3)^4=1,\, s_2^2=3s_1s_3\},\]
with
\[\mathcal{P}_{5,1}=\bigcup_{s_0,s_1\in\mathbb{F}_5}\{s_1x+s_0\colon\, s_1^4=1\}\]
and
\[\mathcal{P}_{5,3}=\bigcup_{s_0,s_1,s_2,s_3\in\mathbb{F}_5}\{s_3x^3+s_2x^2+s_1x+s_0\colon\, s_3^4=1,\, s_2^2=3s_1s_3\}.\]
\end{proposition}

\begin{proof} The reduced Gr\"obner basis of the ideal $I_5\subset \mathbb{F}_5[X_5]$ described in Lemma \ref{lemma_affine}, with respect to the degree inverse lexicographical order, is the set
\[\left\{(x_1+x_3)^4-1,\hspace{0.3cm}
x_2^2+2x_1x_3,\hspace{0.3cm}
x_2(x_1^2+x_3^2),\hspace{0.3cm}
x_3(x_1^3+x_1x_3^2),\hspace{0.3cm}
x_3^5-x_3,\hspace{0.3cm}
x_2(x_3^4-1)\right\}.\]
This is also the reduced Gr\"obner basis of the ideal $\langle\, (x_1+x_3)^4-1,\hspace{0.3cm}x_2^2+2x_1x_3,\hspace{0.3cm} x_2(x_1^2+x_3^2)\,\rangle$, which therefore coincides with $I_5$. The result follows readily from Lemma \ref{lemma_affine} once we observe that the algebraic set of this last ideal coincides with that of the ideal $\langle\, (x_1+x_3)^4-1,\hspace{0.3cm}x_2^2+2x_1x_3\,\rangle$.
\end{proof}

\vspace{0.2cm}

\begin{theorem}\label{theorem_LP4} The set $\mathcal{LP}_4$ is decomposed into the following two disjoint subsets.
\begin{itemize}
    \item The subset $\mathcal{LP}_{4,1}\sqcup \mathcal{LP}_{4,2}$, which is formed by 144 LPPs with separate variables of the form
\begin{equation}\label{eq_S1}
s_{20}x^2+s_{02}y^2+s_{10}x+s_{01}y+s_{00}\in\mathbb{F}_4[x,y]
\end{equation}
where
\[s_{01}^3+s_{02}^3 = s_{10}^3+s_{20}^3 =1.\]

\item The subset $\mathcal{LP}_{4,4}$, which is formed by 432 polynomials of the form
\begin{equation}\label{eq_LP4}
\begin{split}
\left(s_{22}xy+s_{21}x+s_{12}y+s_{12}s_{21}s_{22}^2\right)xy+s_{20}x^2+s_{02}y^2+\\ +(s_{20}s_{22}+s_{21}^2)s_{12}s_{22}x+(s_{02}s_{22}+s_{12}^2)s_{21}s_{22}y+s_{00}\in\mathbb{F}_4[x,y]    
\end{split}
\end{equation}
where
\begin{equation}\label{eq_LP4_constraints}
\begin{cases}
s_{22}^3=s_{21}^3=s_{12}^3=1,\\
s_{20}(s_{20}s_{22}+s_{21}^2)=0,\\
s_{02}(s_{02}s_{22}+s_{12}^2)=0.
\end{cases}
\end{equation}
\end{itemize}
\end{theorem}

\begin{proof} In order to determine the set $\mathcal{LP}_4$, we conveniently replace the generators of the ideal $\overline{I}_4$ in Lemma \ref{lemma_affine} by the polynomials in (\ref{equation_G4}) as we have indicated in Remark \ref{remak_affine}. The reduced Gr\"obner basis of the resulting ideal implies that every polynomial $\sum_{i,j=0}^2 s_{ij}x^iy^j\in \mathcal{LP}_4$ satisfies the next nine constraints.

\[\begin{cases}
\begin{array}{r}
s_{10}^3+s_{20}^3=1,\\
s_{01}^3+s_{02}^3=1,\\
s_{21}^3+s_{22}^3=0,\\
s_{12}^3+s_{22}^3=0,\\
s_{12}s_{10}+s_{11}^2s_{10}^3=0,\\
s_{11}+s_{12}s_{21}s_{22}^2=0,\\
s_{21}+s_{02}^2s_{12}s_{11}+s_{01}^2s_{11}^2=0,\\
s_{22}+s_{01}^2s_{12}s_{11}+s_{02}^2s_{12}^2=0,\\
s_{10}s_{22}+s_{12}s_{20}+s_{12}s_{21}^2s_{22}^2=0.
\end{array}
\end{cases}\]
Depending on whether $s_{22}
=0$ or not, these constraints decompose $\mathcal{LP}_4$ into the two subsets $\mathcal{LP}_{4,2}$ and $\mathcal{LP}_{4,4}$ described in the statement.
\end{proof}

\vspace{0.2cm}

In practice, if one is only interested in the number $N_q(d)$ (that is, in the cardinality of $\mathcal{P}_{q,d}$), and not in the enumeration of the polynomials in $\mathcal{P}_{q,d}$, then we can modify the ideals of Lemma \ref{lemma_affine_degrees} by including $x_0$ as generator of both ideals and replacing the generator $x_d^{q-1}-1$ by $x_d-1$. That is, we focus on monic polynomials without constant term. The cardinalities of the new quotient rings coincide with the required ones but are divided by $q(q-1)$. This is due to the fact that, if $f\in\mathcal{P}_q$, then $af+b\in \mathcal{P}_q$ for all $a,b\in\mathbb{F}_q$, with $a\neq 0$. Based on this algebraic geometry procedure, we show in Table \ref{table_Pqd} the number of polynomials in $\mathcal{P}_{q,d}$, with $q\leq 13$. (Note that the number of polynomials in $\mathcal{P}_8$ having degree less than six is $5,376$, which differs from the value $5,368$ indicated in \cite{Konyagin2002}. The remaining values for the number of polynomials having degree less than $q-2$, for $q\leq 11$, coincide with those indicated in that paper.) 

\begin{table}[ht]
\caption{Explicit values for $N_q(d)$.}\label{table_Pqd}%
\begin{tabular}{rrr||rrr}
$q$ & $d$  & $N_q(d)$ & $q$ & $d$  & $N_q(d)$ \\
\midrule
4   & 1   & 12     &  11 & 1 & 110\\
    & 2   & 12     &     & 3 & 1,210\\
5   & 1  & 20      &     & 6 & 29,040\\
    & 3   & 100    &     & 7 & 272,250\\ 
7   & 1  & 42      &     & 8 & 3,332,340\\
    & 4   & 588    &     & 9 & 36,281,850\\ 
    & 5 & 4,410    &  13 & 1 & 156\\
8  & 1  & 56       &     & 5 & 38,532\\
   & 2  & 56       &     & 7 & 233,220\\
   & 3  & 448      &     & 8 & 2,798,640\\
   & 4  & 1,232    &     & 9 & 33,948,720\\
   & 5  & 3,584    &     & 10 & 442,144,560\\
   & 6  & 34,944   &     & 11 & 5,747,856,972\\
9  & 1  & 72       &\\
   & 3  & 360      &\\
   & 5 & 1,944     &\\
   & 6 & 39,744    & \\
   & 7 & 320,760   &\\
\end{tabular}
\end{table}

\vspace{0.2cm}

The algebraic method described in Lemma \ref{lemma_affine}, Lemma \ref{lemma_affine_degrees} and Remark \ref{remak_affine} can be conveniently adapted to identify particular families of Latin squares, which in turn give rise to equivalent families of LPPs. As an illustrative example, we describe here the ideals associated with symmetric LPPs and reduced LPPs. First, let $\mathcal{SLP}_q$ and $\mathcal{SLP}_{q,d}$ denote, respectively, the subsets of symmetric polynomials in $\mathcal{LP}_q$ and $\mathcal{LP}_{q,d}$. Recall here that a polynomial $f\in\mathbb{F}_q[x,y]$ is called {\em symmetric} if $f(x,y)=f(y,x)$. Then, the next result holds readily from Lemmas \ref{lemma_affine} and \ref{lemma_affine_degrees}.

\begin{proposition}\label{proposition_symmetric}
The algebraic set of the zero-dimensional radical ideal
\[\mathcal{J}_q + \left\langle\,x_{ij}-x_{ji}\colon\, 0\leq i<j\leq q-2\,\right\rangle\subset \mathbb{F}_q[\overline{X}_q].\]
If $\mathcal{J}_q$ is replaced by $\mathcal{J}_{q,d}$, then the resulting zero-dimensional radical ideal is identified with $\mathcal{SLP}_{q,d}$. 
\end{proposition}

\vspace{0.2cm}

Furthermore, the notion of reduced Latin squares readily translates to LPPs as follows. Recall here that a Latin square $L\in\mathcal{L}_q$ is called {\em reduced} if $L[a,0]=L[0,a]=a$ for all $a\in \mathbb{F}_q$. 

\begin{definition}\label{def:reduced} A polynomial $f\in \mathcal{LP}_q$ is called {\em reduced} if $f(x,0)=x$ and $f(0,y)=y$. This is equivalent to say that 
\begin{equation}\label{equation_reduced}
f=g\cdot xy + x + y,
\end{equation}
for some $g\in\mathbb{F}_q[x,y]$. We denote by $\mathcal{RLP}_q$ the subset of reduced polynomials in $\mathcal{LP}_q$.
\end{definition}

\vspace{0.2cm}

This definition is justified by fact that a polynomial $f\in \mathcal{RLP}_q$ if and only if the Latin square $L_f$ is reduced. Then, the next result holds readily from the equivalent relation among Latin squares

\begin{lemma} The number of local permutation polynomials over $\mathbb{F}_q$ is
\[|\mathcal{LP}_q|=q!\cdot (q-1)!\cdot |\mathcal{RLP}_q|.\]
\end{lemma}

\vspace{0.2cm}

From (\ref{eq_LPP2}), (\ref{eq_LPP3}) and (\ref{equation_reduced}), the only reduced polynomial in both $\mathcal{RLP}_2$ and $\mathcal{RLP}_3$ is $x+y$. For $q>3$, the next result holds readily from Lemma \ref{lemma_affine} and Definition \ref{def:reduced}.

\begin{proposition}\label{proposition_reduced} 
The algebraic set of the zero-dimensional radical ideal 
\[\mathcal{J}_q + \left\langle\,x_{00},\,x_{10}-1,\,x_{01}-1\,\right\rangle+\left\langle\,x_{0i},\,x_{i0}\colon\, 2\leq i\leq q-2\,\right\rangle\subset \mathbb{F}_q[\overline{X}_q].\]
is identified with the set $\mathcal{RLP}_q$. If $\mathcal{J}_q$ is replaced by $\mathcal{J}_{q,d}$, then the resulting zero-dimensional radical ideal is identified with $\mathcal{RLP}_{q,d}$. 
\end{proposition}

\begin{table}[ht]
\caption{Cardinality of $\mathcal{LP}_{q,d}$.}\label{table_LPqd}%
\begin{tabular}{llrrr}
$q$ & $d$  & $|\mathcal{LP}_{q,d}|$ & $|\mathcal{SLP}_{q,d}|$ & $|\mathcal{RLP}_{q,d}|$\\
\midrule
4 & 1 & 36 & 12 & 1\\
  & 2 & 108 & 12 & 0\\
  & 4 & 432 & 72 & 3\\
5 & 1 & 80 & 20 & 1\\
  & 3 & 3,200 & 200 & 0\\
  & 5 & 22,000 & 100 & 19\\
  & 6 & 135,920 & 400 & 36\\
\end{tabular}
\end{table}

Based on our algebraic approach, we show in Table \ref{table_LPqd} the cardinality of the sets $\mathcal{LP}_{q,d}$, $\mathcal{SLP}_{q,d}$  and $\mathcal{RLP}_{q,d}$  for $q\in \{4,5\}$. As an illustrative example, we detail this procedure to characterize all polynomials in $\mathcal{SLP}_4$ and $\mathcal{RPL}_4$. 

\begin{theorem}\label{theorem_SLP4} Every polynomial in $\mathcal{SLP}_4$ is of the form
\[(s_{22}xy +s_{21}(x+y)+ s_{11})xy+s_{20}(x^2+y^2)+s_{10}(x+y)+s_{00},\]
with $s_{22},s_{21},s_{20},s_{10},s_{00}\in\mathbb{F}_4$ satisfying the next five constraints.
\[\begin{cases}\begin{array}{r}
s_{10}^3+s_{20}^3=1,\\
s_{11}(s_{22}^3+1)=0,\\
s_{21}(s_{22}^3+1)=0,\\
s_{20}s_{21}+s_{10}s_{22}+s_{22}^2=0,\\
s_{11}s_{21}+s_{22}^2=0.\\
\end{array}
\end{cases}\]
\end{theorem}

\begin{proof} The reduced Gr\"obner basis of the ideal described in Proposition \ref{proposition_symmetric} for $q=4$, with respect to the lexicographical order, is formed by the following 16 generators.
\[x_{12}+x_{21},\hspace{1cm}
x_{02}+x_{20},\hspace{1cm}
x_{01}+x_{10},\hspace{1cm}
x_{10}x_{20},\hspace{1cm}
x_{11}^2+x_{21}x_{22},\]
\[
x_{21}^2+x_{11}x_{22},\hspace{1cm} 
x_{11}x_{21}+x_{22}^2,\hspace{1cm} 
x_{11}(x_{22}^3+1),\hspace{1cm} 
x_{21}(x_{22}^3+1),\hspace{0.8cm} 
x_{10}^3+x_{20}^3+1,\]
\[x_{10}x_{21}(x_{10}+x_{22}),\hspace{2cm} 
x_{10}x_{22}(x_{10}+x_{22}),\hspace{2cm} 
x_{20}x_{21}+x_{10}x_{22}+x_{22}^2,\]
\[x_{11}x_{20}+x_{10}x_{21}+x_{21}x_{22},\hspace{1cm} 
x_{10}x_{11}+x_{11}x_{22}+x_{20}x_{22},\hspace{1cm} 
x_{22}(x_{20}^2+x_{10}x_{21}+x_{21}x_{22}).\]
Their zeros give rise to the constraints of the statement.
\end{proof}

\vspace{0.2cm}

\begin{theorem}\label{theorem_RLP4} Every polynomial in $\mathcal{RLP}_4$ is of the form
\begin{equation}\label{equation_RLP4}
\left(a^3xy+a(x+y)+a^2\right)xy+x+y\in\mathbb{F}_4[x,y]
\end{equation}
with $a\in\mathbb{F}_4$. If $a=0$, then we get the polynomial $x+y\in \mathcal{LP}_{4,1}$. If $a\neq 0$, then we get three reduced polynomials in $\mathcal{LP}_{4,4}$. Moreover, if $u$ denotes a primitive element of $\mathbb{F}_4$, with minimum polynomial $u^2+u+1$, then all the four reduced Latin squares in $\mathcal{L}_4$ are represented by the following array.
\[\begin{array}{|c|c|c|c|} \hline
0 & 1 & u & u^2\\ \hline
1 & a^3+a^2 & (a^3+1)u^2+a^2u+a & (a^3+1)u+a^2u^2+a\\ \hline
u & (a^3+1)u^2+a^2u+a & a^3u+a^2u^2& (a+1)^3\\ \hline
u^2 & (a^3+1)u+a^2u^2+a& (a+1)^3& a^3u^2+a^2u\\ \hline
\end{array}\]
\end{theorem}

\begin{proof} The reduced Gr\"obner basis of the ideal described in Proposition \ref{proposition_reduced} for $q=4$, with respect to the lexicographical order, is the set
{\small\[\left\{
x_{22}^2+x_{22},\,
x_{21}(x_{22}+1),\,
x_{21}^3+x_{22},\,
x_{20},\,
x_{12}+x_{21},\,
x_{11}+x_{21}^2,\,
x_{10}+1,\,
x_{02},\,
x_{01}+1,\,
x_{00}\right\}.\]}
Then, the result follows readily from the zeros of these generators.    
\end{proof}

\vspace{0.5cm}

It is clear that the constraints satisfied by the coefficients of any polynomial in $\mathcal{P}_q$, $\mathcal{LP}_q$ or $\mathcal{RLP}_q$ become much more convoluted as the order $q$ increases. Determining them from this algebraic approach becomes even harder because, as we have already mentioned in the introductory section, the calculation of Gr\"obner bases is extremely sensitive to the number of variables and the length and degree of the generators. In any case, it becomes much easier when some kind of symmetry is imposed in the polynomials under consideration. The next section delves into this aspect by translating the notions of isotopism of Latin square to the context of LPPs.

\section{Isotopisms of LPPs}\label{sec:symmetries}

In the preliminary section, we have shown how isotopisms play a relevant role in studying the symmetries of a Latin square. This notion is naturally translated to quasigroups (see \cite{Albert1943}) and hence to LPPs. This section focuses on this last aspect. Previously, we introduce some notation concerning the composition of PPs and LPPs.

Since every polynomial in $\mathcal{P}_q$ acts as a permutation on $\mathbb{F}_q$, the composition of two polynomials $f,g\in\mathcal{P}_q$ is the polynomial
\begin{equation}\label{eq_fg}
fg:=f(g(x))\in\mathcal{P}_q    
\end{equation}
where the composition is performed modulo $x^q-x$. In particular, the inverse of a polynomial $f\in\mathcal{P}_q$ is the unique polynomial $f^{-1}\in \mathcal{P}_q$ such that $ff^{-1}=f^{-1}f=x$. Furthermore, if $f\in\mathcal{LP}_q$ and $g,h\in\mathcal{P}_q$, then we define the polynomials
\begin{equation}\label{eq_fgh}f[g,h]:=f(g(x),h(y))\in \mathcal{LP}_q
\end{equation}
and
\begin{equation}\label{eq_gf}
gf:=g(f(x,y))\in\mathcal{LP}_q,
\end{equation}
where the compositions are performed modulo $x^q-x$ and $y^q-y$. Here, we use the bracket notation in (\ref{eq_fgh}) to avoid confusion, because we have, for example, $f[x,x]:=f(x,y)$. In addition, by abuse of notation, we assume that both the polynomial $h$ in (\ref{eq_fgh}) and the polynomial $g$ in (\ref{eq_gf}) can act on the variable $y$ even if $g,h\in\mathbb{F}_q[x]$. Both products (\ref{eq_fgh}) and (\ref{eq_gf}) were already considered by Mullen  \cite{Mullen1981} and Lee and Ko \cite{Lee1994} to construct LPPs. The last two authors raised a multivariate version of the following problem.

\begin{problem}\label{problem_Lee} Find some polynomial in $\mathcal{LP}_q$ that cannot be constructed from a polynomial with separate variables (that is, a polynomial $h_1(x)+h_2(y)\in\mathcal{LP}_q$, with $h_1,h_2\in\mathcal{P}_q$) and the application of the products (\ref{eq_fgh}) and (\ref{eq_gf}).    
\end{problem}

\vspace{0.2cm}

In what follows, we show how this question is related to the notion of isotopism of Latin squares, which is readily translated to LPPs as follows.

\begin{definition}\label{definition_isotopic} A polynomial $f\in \mathcal{LP}_q$ is said to be {\em isotopic} to a polynomial $g\in \mathcal{LP}_q$ if 
\begin{equation}\label{eq_isotopism}
g[h_1,\, h_2]=h_3f
\end{equation}
for some $h_1,\,h_2,\,h_3\in \mathcal{P}_q$. The triple $(h_1,h_2,h_3)$ is an {\em isotopism} from $f$ to $g$. This is {\em principal} if $h_3=x$. If this is the case, then $f$ and $g$ are said to be {\em principal isotopic}. Furthermore, if $h_1=h_2=h_3$, then we get an {\em isomorphism} from $f$ to $g$, which are then said to be {\em isomorphic}. 
\end{definition}

\vspace{0.5cm}

These notions are justified by the next result, which holds from (\ref{eq_LS_isotopic}) and (\ref{eq_isotopism}).

\begin{lemma}\label{lemma_isotopic} Two polynomials $f,g\in \mathcal{LP}_q$ are (principal) isotopic (respectively, isomorphic) if and only if their associated Latin squares $L_f, L_g\in \mathcal{L}_q$ are (principal) isotopic (respectively, isomorphic).  
\end{lemma}

\vspace{0.2cm}

Since every polynomial in $\mathcal{P}_q$ acts as a permutation on $\mathbb{F}_q$, it is readily verified that being isotopic is an equivalence relation among LPPs. This allows their distribution in {\em isotopism classes}. A relevant isotopism class of LPPs is that one containing the linear polynomial $x+y$. In order to deal with this class, we introduce here the notion of {\em isolinearity}.

\begin{definition}\label{definition_isolinear} An LPP  is called {\em isolinear} if it is isotopic to $x+y$. Otherwise, it is called {\em non-isolinear}.
\end{definition}

\vspace{0.2cm}

The next lemma characterizes the isolinear polynomials in $\mathcal{LP}_q$. Then, the subsequent proposition shows that the polynomial sought in Problem \ref{problem_Lee} is any non-isolinear.

\begin{lemma}\label{lemma_isotopic_23} A polynomial $f\in\mathcal{LP}_q$ is isolinear if and only if 
\[f:=f(x,y)=h_3(h_1(x)+h_2(y))\]
for some polynomials $h_1,h_2,h_3\in \mathcal{P}_q$. In particular, every LPP with separate variables is isolinear.
\end{lemma}

\begin{proof} The result holds readily from Definitions \ref{definition_isotopic} and \ref{definition_isolinear}. Then, the consequence follows once we consider the trivial permutation $h_3=x$.
\end{proof}

\vspace{0.2cm}

\begin{proposition} An LPP is a solution of Problem \ref{problem_Lee} if and only if it is non-isolinear.    
\end{proposition}

\begin{proof} On the one hand, if $f\in\mathcal{LP}_q$ is a polynomial with separate variables, then the product $f[g,h]\in\mathcal{LP}_q$ defined in (\ref{eq_fgh}) is also a polynomial with separate variables. On the other hand, Lemma \ref{lemma_isotopic_23} implies that applying the product (\ref{eq_gf}) to a polynomial with separate variables gives rise to an isolinear polynomial. Then, the result holds from the fact that applying the products (\ref{eq_fgh}) and (\ref{eq_gf}) to an isolinear polynomial gives rise to an isotopic polynomial.    
\end{proof}

\vspace{0.2cm}

The next result deals with the isolinear polynomials in $\mathcal{LP}_q$, with $q\in\{2,3,4\}$.

\begin{proposition}\label{proposition_isolinear} Every polynomial in $\mathcal{LP}_2$ and $\mathcal{LP}_3$ is isolinear. Moreover, a polynomial in $\mathcal{LP}_4$ is isolinear if and only if it belongs to the set $\mathcal{LP}_{4,1}\cup \mathcal{LP}_{4,2}$ described in Theorem \ref{theorem_LP4}.
\end{proposition}

\begin{proof} From (\ref{eq_LPP2}), (\ref{eq_LPP3}) and (\ref{eq_S1}), every polynomial in any of the sets $\mathcal{LP}_2$, $\mathcal{LP}_3$ and $\mathcal{LP}_{4,2}$ have separate variables. Hence, the first statement and the sufficient condition of the second one follow readily from Lemma \ref{lemma_isotopic_23}. In order to prove the necessary condition, let $f\in\mathcal{LP}_4$ be isolinear. From Lemma \ref{lemma_isotopic_23}, there exist three polynomials $h_1,h_2,h_3\in \mathcal{P}_q$ such that $f=h_3(h_1(x)+h_2(y))$. In addition, Proposition \ref{proposition_P4} implies that $h_3=s_2x^2+x_1+s_0$ for some $s_0,s_1,s_2\in\mathbb{F}_4$ such that $s_1^3+s_2^3=1$. Thus, since the base field has characteristic two,
\[f=s_2(h_1^2(x)+h_2^2(y))+s_1(h_1(x)+h_2(y))+s_0.\]
That is, $f$ has separate variables, and hence, $f\in\mathcal{LP}_{4,2}$.    
\end{proof}

\vspace{0.5cm}

The equivalence among isotopisms of Latin squares and LPPs allows for the translation of known results of the former to the latter. In this regard, the next result readily follows from Lemma \ref{lemma_isotopic} and the fact that every Latin square is principal isotopic to a reduced Latin square by a convenient permutation of its rows and columns. 

\begin{proposition}\label{proposition_isotopic_reduced} Every LPP is principal isotopic to a reduced LPP. 
\end{proposition}

\vspace{0.2cm}

This result plays a relevant role in describing all polynomials in $\mathcal{LP}_q$ from those in $\mathcal{P}_q$ and $\mathcal{RLP}_q$. The following theorem illustrates this fact for $q=4$. It follows readily from Proposition \ref{proposition_isotopic_reduced}, together with Proposition \ref{proposition_P4} and (\ref{equation_RLP4}).

\begin{theorem}\label{theorem_LP} Every polynomial in $\mathcal{LP}_4$ is of the form $f[h,h']$, where
\[\begin{cases}
f:=\left(a^3xy+a(x+y)+a^2\right)xy+x+y\in\mathcal{RLP}_4,\\
h:=bx^2+cx+d\in\mathcal{LP}_4,\\
h':=b'x^2+c'x+d'\in\mathcal{LP}_4.
\end{cases}\]
Here, $b^3+c^3=1={b'}^3+{c'}^3$.
\end{theorem}

\vspace{0.2cm}

Proposition \ref{proposition_isotopic_reduced} also facilitates the distribution of LPPs into isotopism classes. To do it, we associate each LPP under consideration with a reduced LPP to which the former is principal isotopic. Then, we prove whether the resulting polynomials are isotopic. The following result proposes a construction in this regard.

\begin{proposition}\label{proposition_isotopic_reduced_2} Let $f\in \mathcal{LP}_q$  and $a,b\in \mathbb{F}_q$ be such that $f(a,b)=0$. Then, $f$ is principal isotopic to the reduced polynomial 
\[\rho_{f,a,b}:=f\left[p_{a,b}+a,q_{a,b}+b\right]\in \mathcal{RLP}_q,\]
where  $p_{f,a,b}:=\left(f(x+a,b)\right)^{-1}\in\mathcal{P}_q$ and $q_{f,a,b}:=\left(f(a,x+b)\right)^{-1}\in\mathcal{P}_q$.
\end{proposition}

\begin{proof} Definition \ref{definition_isotopic} ensures that $\rho_f$ is principal isotopic to $f$ by means of the principal isotopism $\left(p_{f,a,b}(x)+a,\,q_{f,a,b}(x)+b,\,x\right)\in\mathcal{P}_q\times\mathcal{P}_q\times\mathcal{P}_q$. Now, we prove that $\rho_f$ is reduced. Developing in powers of $x$ and $y$,  we have
\[f(x+a,y+b)=xyh(x,y)+f(x+a,b)+f(a,y+b)\]
for some polynomial $h\in\mathbb{F}_q[x,y]$. Then,
\[\rho_{f,a,b}=p_{f,a,b}(x)q_{f,a,b}(y)h(p_{f,a,b}(x),q_{f,a,b}(y))+x+y.\]
This is a reduced a polynomial, because $p_{f,a,b}(0)=q_{f,a,b}(0)=0$ and hence, $x|p_{f,a,b}(x)$ and $y|q_{f,a,b}(y)$.
\end{proof}

\vspace{0.2cm}

The following result illustrates this procedure by showing that every polynomial in the subset $\mathcal{LP}_{4,4}\subset \mathcal{LP}_4$ described in Theorem \ref{theorem_LP4} belongs to the same isotopism class. Of course, this result holds readily from Lemma \ref{lemma_isotopic}, Proposition \ref{proposition_isolinear} and the fact that $\mathcal{L}_4$ has exactly two isotopism classes. Nevertheless, in order to show the potential of our algebraic approach, we give an alternative proof based on Proposition \ref{proposition_isotopic_reduced_2}.

\begin{theorem}\label{theorem_S2} Every polynomial in $\mathcal{LP}_{4,4}$ belongs to the same isotopism class.
\end{theorem}

\begin{proof} From Proposition \ref{proposition_P4}, every polynomial in $\mathcal{P}_4$ is of the form  $ax+b$ or $ax^2+b$, with $a\in\mathbb{F}_4\setminus\{0\}$ and $b\in\mathbb{F}_4$. In the first case, its compositional inverse is $a^2(x+b)$. Otherwise, it is $a(x+b)^2$. From (\ref{eq_LP4}) and (\ref{eq_LP4_constraints}), every polynomial in $\mathcal{LP}_{4,4}$ is of the form

\[f_1:=\left(axy+bx+cy+a^2bc\right)xy+ab^2cx+abc^2y+d,\]
\[f_2:=\left(axy+bx+cy+a^2bc\right)xy+ab^2cx+a^2c^2y^2+d,\]
\[f_3:=\left(axy+bx+cy+a^2bc\right)xy+a^2b^2x^2+abc^2y+d,\]
or
\[f_4:=\left(axy+bx+cy+a^2bc\right)xy+a^2b^2x^2+a^2c^2y^2+d,\]
with $a,b,c\in\mathbb{F}_4\setminus\{0\}$ and $d\in\mathbb{F}_4$. For each positive integer $i\leq 4$, let $a_i\in\mathbb{F}_4$ be such that $f_i(a_i,0)=0$. That is, $a_1=a_2=a^2bc^2d$ and $a_3=a_4=a^2b^2d^2$. Then,
\[p_{f_i,a_i,0}:=\left(f_i(x+a_i,0)\right)^{-1}=\begin{cases}
    \begin{array}{ll}
    a^2bc^2x, & \text{ if } i\in\{1,2\},\\
    a^2b^2x^2, & \text{ if } i\in\{3,4\}.   
    \end{array}
\end{cases}\]
In addition,
\begin{align*}f_i(a_i,x) & =\begin{cases}
    \begin{array}{ll}
    a^2bd(bcd+1)x^2+a(b^2d+bc^2+cd^2)x, & \text{ if } i\in\{1,3\},\\
    a^2(b^2cd^2+bd+c^2)x^2+ad(b^2+cd)x
, & \text{ if } i\in\{2,4\}.
    \end{array}
\end{cases}\\
& = \begin{cases}\begin{array}{ll}
    abc^2x, & \text{ if } \begin{cases}i\in\{1,3\} \text{ and } d=bcd^2,\\
    i\in\{2,4\} \text{ and } bcd(bcd+1)=1,
    \end{cases}\\
    a^2c^2x^2, & \text{ if } \begin{cases}
     i\in\{1,3\} \text{ and } bcd(bcd+1)=1,\\
     i\in\{2,4\} \text{ and } d=bcd^2.
    \end{cases}
    \end{array}
    \end{cases}
\end{align*}
Hence,
\[q_{f_i,a_i,0}:=\left(f_i(a_i,x)\right)^{-1}= \begin{cases}\begin{array}{ll}
    a^2b^2cy, & \text{ if } \begin{cases}i\in\{1,3\} \text{ and } d=bcd^2,\\
    i\in\{2,4\} \text{ and } bcd(bcd+1)=1,
    \end{cases}\\
    a^2c^2y^2, & \text{ if } \begin{cases}
     i\in\{1,3\} \text{ and } bcd(bcd+1)=1,\\
     i\in\{2,4\} \text{ and } d=bcd^2.
    \end{cases}
    \end{array}
    \end{cases}\]
In any case, 
\[\rho_{f_i,a_i,0}:=(xy+b^2c^2(x+y)+bc)xy+x+y\]
Every polynomial in $\mathcal{LP}_{4,4}$ is therefore principal isotopic to a reduced polynomial
\[g_\alpha:=(xy+\alpha(x+y)+\alpha^2)xy+x+y\in\mathcal{LP}_{4,4},\]
with $\alpha\in\mathbb{F}_4\setminus\{0\}$. Particularly, if $\alpha\neq 1$, then $g_\alpha$ is isomorphic to $g_1$ by means of the polynomial $h_\alpha=\alpha x^2$, because
\[g_\alpha(h_\alpha,h_\alpha)=\alpha\left((xy+x+y+1)xy+x^2+y^2\right)=h_\alpha(g_1).\]
As a consequence, every polynomial in $\mathcal{LP}_{4,4}$ is isotopic to $g_1$.
\end{proof}

\vspace{0.2cm}

Theorem \ref{theorem_S2}  can be directly proved by defining an isotopism from each of the four polynomials $f_1,f_2,f_3,f_4\in\mathcal{LP}_{4,4}$ to the polynomial $g_1\in\mathcal{LP}_{4,4}$, all of them described in the previous proof. In this regard, it is straightforward to see that 
\[f_i\left(a^2cx^{e_i},a^2by^{f_i}\right)=b^2c^2g_1(x,y)+d\]
for each positive integer $i\leq 4$, where 
\[(e_1,f_1)=(1,1),\quad (e_2,f_2)=(1,2),\quad (e_3,f_3)=(2,1) \quad \text{ and } \quad (e_4,f_4)=(2,2).\]

\vspace{0.2cm}

The next result is an immediate consequence of Proposition \ref{proposition_isolinear} and Theorem \ref{theorem_S2}.

\begin{corollary}\label{corollary_S1S2} The set $\mathcal{LP}_4$ is formed by two isotopism classes, which are respectively described by the the sets $\mathcal{LP}_{4,2}$ and $\mathcal{LP}_{4,4}$ defined in Theorem \ref{theorem_LP4}.
\end{corollary}

\section{Complete mappings of LPPs}\label{sec:Orthomorphism}

We finish our study by dealing with the natural translation to LPPs of complete mappings and orthomorphisms of quasigroups, showing their relationship with transversals and isotopisms of Latin squares. First, for each polynomial $f\in\mathcal{LP}_q$ and any pair of polynomials $g_1,g_2\in\mathcal{P}_q$, we define the polynomial
\begin{equation}\label{eq_hqp}
\odot_f(g_1,g_2):=f(g_1(x),g_2(x))\in\mathbb{F}_q[x],
\end{equation}
where the composition is performed modulo $x^q-x$. This notation follows from the definition of the Hadamard quasigroup product described in (\ref{eq_hqp_multiary}). The following preliminary result describes the relationship among the compositions described in (\ref{eq_fg}--\ref{eq_gf}) and (\ref{eq_hqp}).

\begin{lemma}\label{lemma_hqp_comp}  Let $f\in\mathcal{LP}_q$ and $g_1,g_2,h_1,h_2,h\in\mathcal{P}_q$. The following statements hold.
\begin{enumerate}
    \item If $h_1,h_2\in\mathcal{P}_q$, then
\[\odot_{f[g_1,g_2]}(h_1,h_2)=\odot_f\left(g_1h_1,\,g_2h_2\right).\]
\item If $h\in\mathcal{P}_q$, then
\[h\odot_f(g_1,g_2)=\odot_{hf}(g_1,g_2)\]
and
\[\odot_f(g_1,g_2)h=\odot_{f}(g_1h,g_2h).\]
\end{enumerate}
\end{lemma}

\begin{proof} Let $a\in \mathbb{F}_q$. The first statement holds because
\begin{align*}
\odot_{f[g_1,g_2]}(h_1,h_2)(a) & =f[g_1,g_2](h_1(a),\,h_2(a))=f\left(g_1(h_1(a)),g_2(h_2(a))\right)=\\
&=f\left(g_1h_1(a),g_2h_2(a)\right)=\odot_f\left(g_1h_1,\,g_2h_2\right)(a).
\end{align*}
The second statement holds because
\begin{align*}h\odot_f(g_1,g_2)(a)&=h\left(\odot_f(g_1,g_2)(a)\right)=h\left(f(g_1(a),g_2(a))\right)=\\
& =hf(g_1(a),g_2(a))=\odot_{hf}(g_1,g_2)(a) 
\end{align*}
and
\[\odot_f(g_1,g_2)h(a)=\odot_f(g_1,g_2)(h(a))=f(g_1(h(a)),g_2(h(a)))=\odot_{f}(g_1h,g_2h)(a).\]
\end{proof}

\vspace{0.2cm}

The next result establishes when $\odot_f(g_1,g_2)\in\mathcal{P}_q$ by means of the transversals of the Latin square $L_f\in\mathcal{L}_q$. It constitutes the natural translation to LPPs of Lemma 4.2 in \cite{Falcon2025}. 

\begin{lemma}\label{lemma_hqp} Let $f\in\mathcal{LP}_q$ and $g_1,g_2\in\mathcal{P}_q$. Then, $\odot_f(g_1,g_2)\in\mathcal{P}_q$ if and only if the subset
\[\left\{(g_1(a),\,g_2(a),\,\odot_f(g_1,g_2)(a)\right\}\subset \mathrm{Ent}(L_f)\]
is a transversal of $L_f$.
\end{lemma}

\begin{proof} The polynomial $\odot_f(g_1,g_2)$ is permutational if and only if, for each pair of distinct elements $a,b\in\mathbb{F}_q$, 
\[L_f[g_1(a),g_2(a)]=\odot_f(g_1,g_2)(a)\neq \odot_f(g_1,g_2)(b)=L_f[g_1(b),g_2(b)].\]
Then, the result follows readily from the fact that $g_1,g_2\in \mathcal{P}_q$.    
\end{proof}

\vspace{0.2cm}

The case $g_1=x$ generalizes for LPPs both notions of complete mapping and orthomorphism of quasigroups, which were recalled in the preliminary section. From Lemma \ref{lemma_hqp} , and similarly to quasigroups (see Lemma \ref{lemma_transversals}), they are equivalent to transversals of Latin squares. 

\begin{definition}\label{definition_orthomorphism} Let $f\in\mathcal{LP}_q$. A polynomial $g\in\mathcal{P}_q$ is called a {\em complete mapping} of $f$ if $\odot_f(x,g)\in \mathcal{P}_q$. This last polynomial is called an {\em orthomorphism} of $f$.
\end{definition}

\vspace{0.2cm}

The next result shows the relationship among the complete mappings of two isotopic LPPS.

\begin{proposition} \label{proposition_transversal_isotopic} There is a one-to-one correspondence between the sets of complete mappings of any pair of isotopic LPPs. More precisely, if $(h_1,h_2,h_3)$ is an isotopism from a polynomial $f_1\in\mathcal{LP}_q$ to a polynomial $f_2\in\mathcal{LP}_q$, then $g\in\mathcal{P}_q$ is a complete mapping of $f_1$ if and only if $h_2gh_1^{-1}\in\mathcal{P}_q$ is a complete mapping of $f_2$.
\end{proposition}

\begin{proof} Let $g\in\mathcal{P}_q$.  From (\ref{eq_isotopism}) and Lemma \ref{lemma_hqp_comp}, we have
\[h_3\odot_{f_1}(x,g)=\odot_{h_3f_1}(x,g)=\odot_{f_2[h_1,h_2]}(x,g)=\odot_{f_2}(h_1,h_2g)=\odot_{f_2}\left(x,h_2gh_1^{-1}\right)h_1.\]
Then, the result holds because we have from (\ref{eq_fg}) that
\[\odot_{f_1}(x,g)\in\mathcal{P}_q \Leftrightarrow \odot_{f_2}\left(x,h_2gh_1^{-1}\right)h_1=h_3\odot_{f_1}(x,g)\in\mathcal{P}_q \Leftrightarrow \odot_{f_2}\left(x,h_2gh_1^{-1}\right)\in\mathcal{P}_q.\]
\end{proof}

\vspace{0.2cm}

This last result, together with Proposition \ref{proposition_isotopic_reduced}, implies that determining all complete mappings in $\mathcal{LP}_q$ is derived from those of a reduced LPP of each isotopism class. We illustrate this procedure for $q=4$.

\begin{theorem}\label{theorem_cm4} Each polynomial in $\mathcal{LP}_{4,2}$ has eigth complete mappings, whereas no complete mapping exists for any polynomial in $\mathcal{LP}_{4,4}$.
\end{theorem}

\begin{proof} We consider the reduced polynomials
\[f_1:=x+y\in\mathcal{LP}_{4,2} \hspace{1cm}\text{and}\hspace{1cm} f_2:=(xy+x+y+1)xy+x+y\in\mathcal{LP}_{4,4}.\]
We also consider a polynomial $g:=s_2x^2+s_1x+s_0\in\mathcal{P}_4$. Then, 
\[\odot_{f_1}(x,g)=s_2x^2+(s_1+1)x+s_0\]
and
\[\odot_{f_2}(x,g)=(s_2^2+s_1^2+s_2+s_1)x^3 + (s_2^2+s_0^2+s_2+s_1+s_0)x^2 + (s_1^2+s_0^2+s_2+s_1+s_0+1)x + s_0.\]
From Proposition \ref{proposition_P4},
\[g\in\mathcal{P}_4\Leftrightarrow s_2^3+s_1^3=1.\] 
Then, the same result also implies that
\[\odot_{f_1}(x,g)\in\mathcal{P}_4\Leftrightarrow 
\begin{cases}
s_2^3+s_1^3=1,\\
s_2^3+(s_1+1)^3=1.    
\end{cases}\Leftrightarrow 
\begin{cases}
s_2=0,\\
s_1^2+s_1=1.   
\end{cases}\]
and
\begin{align*}
 \odot_{f_2}(x,g)\in\mathcal{P}_4 & \Leftrightarrow \begin{cases}
s_2^3+s_1^3=1,\\
s_2^2+s_1^2+s_2+s_1=0,\\
(s_2^2+s_0^2+s_2+s_1+s_0)^3+(s_1^2+s_0^2+s_2+s_1+s_0+1)^3=1.
\end{cases}\\
 & \Leftrightarrow \begin{cases}
s_2^3+s_1^3=1,\\
s_2^2+s_1^2+s_2+s_1=0,\\
s_1^3+ s_2=0.
\end{cases}
\end{align*}
This last system of equations has no solution, so there is no complete mapping of $f_2$. Then, the result follows straightforwardly from Proposition \ref{proposition_transversal_isotopic} and Corollary \ref{corollary_S1S2}.
\end{proof}

\vspace{0.2cm}

We finish our study by dealing with the generalization of the product (\ref{eq_hqp}) to
\[\odot_f(g_1,g_2):=f(g_1(x,y),g_2(x,y))\in\mathbb{F}_q[x,y],\]
with $f,g_1,g_2\in\mathcal{LP}_q$. In the context of Latin squares, this is equivalent to the Hadamard quasigroup product described in \cite{Falcon2023}. Lee and Ko already asked in \cite{Lee1994} if $\odot_f(g_1,g_2)\in\mathcal{LP}_q$. They showed that this condition holds whenever $f$, $g_1$ and $g_2$ are of degree at most two. In addition, they proved in \cite[Theorem 8]{Lee1994} that $ax$ is a complete mapping of $f$ for some $a\in\mathbb{F}_q\setminus\{0\}$ if and only if $\odot_f(g_1,ag_1)\in\mathcal{LP}_q$. Here, $ag_1:=a\cdot g_1(x,y)$. More generally, the following result constitutes the natural translation to LPPs of Lemma 12 in \cite{Falcon2023}. It follows readily from Lemma \ref{lemma_hqp}.

\begin{proposition}\label{proposition_hqp2} Let $f,g_1,g_2\in\mathcal{LP}_q$. Then, $\odot_f(g_1,g_2)\in\mathcal{LP}_q$ if and only if, for each $a\in\mathbb{F}_q$, both subsets
\[\left\{(g_1(a,b),\,g_2(a,b),\,\odot_f(g_1,g_2)(a,b)\colon\,b\in\mathbb{F}_q\right\}\subset \mathrm{Ent}(L_f)\]
and
\[\left\{(g_1(b,a),\,g_2(b,a),\,\odot_f(g_1,g_2)(b,a)\colon\,b\in\mathbb{F}_q\right\}\subset \mathrm{Ent}(L_f)\]
are transversals of $L_f$.
\end{proposition}

\section{Conclusion and further work}

Despite the natural translation among local permutation polynomials in two variables and the set of Latin squares, there are only some preliminary results in the literature concerning this subject. This paper has delved into this topic by emphasizing how a given family of Latin squares can be represented by a single polynomial. We have described an algebraic geometry procedure to identify the constraints satisfied by the coefficients of this polynomial, and we have shown how this procedure facilitates the enumeration and classification of Latin squares, together with the study of some of their properties. We have also introduced the natural translation to LPPs of some classical notions in the theory of Latin squares like isotopisms, complete mappings and orthomorphisms. We have illustrated all these aspects for LPPS over $\mathbb{F}_4$.

Similarly to Latin squares, which can also be distributed into conjugate and main classes, a comprehensive study is required concerning the natural translation to LPPs of conjugation and paratopisms of Latin squares, which we introduce just here.

\begin{definition}\label{definition_conjugate} Two polynomials $f,g\in\mathcal{LP}_q$ are called {\em conjugate} if one of the following six identities holds.
 \begin{enumerate}
\item $g(x,y)=f(x,y)$.

\item $g(x,y)=f(y,x)$.

\item $g(x,f(x,y))=y$.

\item $g(f(x,y),x)=y$.

\item $g(f(x,y),y)=x$.

\item $g(y,f(x,y),x)=x$.
\end{enumerate}
If all the six identities hold for $f = g$, then the polynomial $f$ is called {\em totally symmetric}.   
\end{definition}

\vspace{0.2cm}

All the notions and results considered in this paper may naturally be generalized to LPPs in more variables, which are equivalent to multiary quasigroups and Latin hypercubes. Some preliminary results on this topic have already been described in \cite{Diestelkamp2004, Gutierrez2023, Gutierrez2025,Gutierrez2025a, Mullen1980a}, but a comprehensive study is required in this regard.

\bibliography{sn-bibliography}

\end{document}